\documentclass[11pt,leqno]{article}
\usepackage[margin=1in]{geometry} 
\usepackage{amssymb,amsfonts,amsmath,bbm,mathrsfs,stmaryrd,mathtools}
\usepackage{xcolor}
\usepackage{url}
\usepackage{dsfont}

\usepackage{extarrows}

\usepackage[shortlabels]{enumitem}
\usepackage{tensor}

\usepackage{xr}
\usepackage[T1]{fontenc}

\usepackage[colorlinks,
linkcolor=black!75!red,
citecolor=blue,
pdftitle={},
pdfproducer={pdfLaTeX},
pdfpagemode=None,
bookmarksopen=true,
bookmarksnumbered=true,
backref=page]{hyperref}

\usepackage{tikz}
\usetikzlibrary{arrows,calc,decorations.pathreplacing,decorations.markings,intersections,shapes.geometric,through,fit,shapes.symbols,positioning,decorations.pathmorphing}

\usepackage{braket}

\usepackage[amsmath,thmmarks,hyperref]{ntheorem}
\usepackage{cleveref}

\newcommand\nthalias[1]{\AddToHook{env/#1/begin}{\crefalias{lemma}{#1}}}

\nthalias{definition}
\nthalias{example}
\nthalias{examples}
\nthalias{remark}
\nthalias{remarks}
\nthalias{convention}
\nthalias{notation}
\nthalias{construction}
\nthalias{theoremN}
\nthalias{propositionN}
\nthalias{corollaryN}
\nthalias{lemma}
\nthalias{proposition}
\nthalias{corollary}
\nthalias{theorem}
\nthalias{conjecture}
\nthalias{question}
\nthalias{assumption}

\creflabelformat{enumi}{#2#1#3}

\crefname{section}{Section}{Sections}
\crefformat{section}{#2Section~#1#3} 
\Crefformat{section}{#2Section~#1#3} 

\crefname{subsection}{\S}{\S\S}
\AtBeginDocument{%
  \crefformat{subsection}{#2\S#1#3}%
  \Crefformat{subsection}{#2\S#1#3}% 
}

\crefname{subsubsection}{\S}{\S\S}
\AtBeginDocument{%
  \crefformat{subsubsection}{#2\S#1#3}%
  \Crefformat{subsubsection}{#2\S#1#3}% 
}

\theoremstyle{plain}

\newtheorem{lemma}{Lemma}[section]
\newtheorem{proposition}[lemma]{Proposition}
\newtheorem{corollary}[lemma]{Corollary}
\newtheorem{theorem}[lemma]{Theorem}

\theoremstyle{plain}
\theoremnumbering{Alph}

\theoremstyle{plain}
\theorembodyfont{\upshape}
\theoremsymbol{\ensuremath{\blacklozenge}}

\newtheorem{definition}[lemma]{Definition}
\newtheorem{example}[lemma]{Example}

\newtheorem{remark}[lemma]{Remark}

\crefname{definition}{definition}{definitions}
\crefformat{definition}{#2definition~#1#3} 
\Crefformat{definition}{#2Definition~#1#3} 

\crefname{ex}{example}{examples}
\crefformat{example}{#2example~#1#3} 
\Crefformat{example}{#2Example~#1#3} 

\crefname{exs}{example}{examples}
\crefformat{examples}{#2example~#1#3} 
\Crefformat{examples}{#2Example~#1#3} 

\crefname{remark}{remark}{remarks}
\crefformat{remark}{#2remark~#1#3} 
\Crefformat{remark}{#2Remark~#1#3} 

\crefname{remarks}{remark}{remarks}
\crefformat{remarks}{#2remark~#1#3} 
\Crefformat{remarks}{#2Remark~#1#3} 

\crefname{convention}{convention}{conventions}
\crefformat{convention}{#2convention~#1#3} 
\Crefformat{convention}{#2Convention~#1#3} 

\crefname{notation}{notation}{notations}
\crefformat{notation}{#2notation~#1#3} 
\Crefformat{notation}{#2Notation~#1#3} 

\crefname{table}{table}{tables}
\crefformat{table}{#2table~#1#3} 
\Crefformat{table}{#2Table~#1#3}

\crefname{lemma}{lemma}{lemmas}
\crefformat{lemma}{#2lemma~#1#3} 
\Crefformat{lemma}{#2Lemma~#1#3} 

\crefname{proposition}{proposition}{propositions}
\crefformat{proposition}{#2proposition~#1#3} 
\Crefformat{proposition}{#2Proposition~#1#3} 

\crefname{propositionN}{proposition}{propositions}
\crefformat{propositionN}{#2proposition~#1#3} 
\Crefformat{propositionN}{#2Proposition~#1#3} 

\crefname{corollary}{corollary}{corollaries}
\crefformat{corollary}{#2corollary~#1#3} 
\Crefformat{corollary}{#2Corollary~#1#3} 

\crefname{corollaryN}{corollary}{corollaries}
\crefformat{corollaryN}{#2corollary~#1#3} 
\Crefformat{corollaryN}{#2Corollary~#1#3} 

\crefname{theorem}{theorem}{theorems}
\crefformat{theorem}{#2theorem~#1#3} 
\Crefformat{theorem}{#2Theorem~#1#3} 

\crefname{theoremN}{theorem}{theorems}
\crefformat{theoremN}{#2theorem~#1#3} 
\Crefformat{theoremN}{#2Theorem~#1#3} 

\crefname{enumi}{}{}
\crefformat{enumi}{#2#1#3}
\Crefformat{enumi}{#2#1#3}

\crefname{assumption}{assumption}{Assumptions}
\crefformat{assumption}{#2assumption~#1#3} 
\Crefformat{assumption}{#2Assumption~#1#3} 

\crefname{construction}{construction}{Constructions}
\crefformat{construction}{#2construction~#1#3} 
\Crefformat{construction}{#2Construction~#1#3} 

\crefname{question}{question}{Questions}
\crefformat{question}{#2question~#1#3} 
\Crefformat{question}{#2Question~#1#3} 

\crefname{equation}{}{}
\crefformat{equation}{(#2#1#3)} 
\Crefformat{equation}{(#2#1#3)}

\numberwithin{equation}{section}

\theoremstyle{nonumberplain}
\theoremsymbol{\ensuremath{\blacksquare}}

\newtheorem{proof}{Proof}
\newcommand\pf[1]{\newtheorem{#1}{Proof of \Cref{#1}}}

\newcommand\bC{{\mathbb C}}

\newcommand\bG{{\mathbb G}}

\newcommand\bR{{\mathbb R}}
\newcommand\bS{{\mathbb S}}

\newcommand\bZ{{\mathbb Z}}

\newcommand\cD{{\mathcal D}}

\newcommand\cG{{\mathcal G}}

\newcommand\cL{{\mathcal L}}
\newcommand\cM{{\mathcal M}}

\newcommand\cO{{\mathcal O}}

\newcommand\cS{{\mathcal S}}

\newcommand\cV{{\mathcal V}}

\newcommand\cX{{\mathcal X}}
\newcommand\cY{{\mathcal Y}}

\DeclareMathOperator{\Ad}{Ad}
\DeclareMathOperator{\id}{id}

\DeclareMathOperator{\End}{\mathrm{End}}
\DeclareMathOperator{\Hom}{\mathrm{Hom}}

\DeclareMathOperator{\spn}{\mathrm{span}}

\DeclareMathOperator{\ev}{\mathrm{ev}}

\DeclareMathOperator{\tr}{tr}

\DeclareMathOperator{\db}{\mathrm{db}}

\DeclareMathOperator{\D}{D}

\DeclareMathOperator{\G}{G}

\DeclareMathOperator{\U}{U}

\DeclareMathOperator{\Qut}{Qut}

\DeclareMathOperator{\Tr}{Tr}

\newcommand{\cat}[1]{\textsc{#1}}

\newcommand{\qedhere}{\mbox{}\hfill\ensuremath{\blacksquare}}

\newcommand{\xrightarrowdbl}[2][]{%
  \xrightarrow[#1]{#2}\mathrel{\mkern-14mu}\rightarrow
}

\title{Quantum-rigid random quantum graphs}
\author{Alexandru Chirvasitu, Piotr M. So{\l}tan and Mateusz Wasilewski}

\begin{document}

\date{}

\newcommand{\Addresses}{{% additional braces for segregating \footnotesize
  \bigskip
  \footnotesize

  \textsc{Department of Mathematics, University at Buffalo}
  \par\nopagebreak
  \textsc{Buffalo, NY 14260-2900, USA}  
  \par\nopagebreak
  \textit{E-mail address}: \texttt{achirvas@buffalo.edu}
  
\medskip  
  
   \textsc{Department of Mathematical Methods in Physics, Faculty of Physics, University of Warsaw}
  \par\nopagebreak
  \textsc{ul.~Pasteura 5, 02-093 Warsaw, Poland}
  \par\nopagebreak
  \textit{E-mail address}: \texttt{piotr.soltan@fuw.edu.pl}  
  
\medskip  
  
   \textsc{Institute of Mathematics, Polish Academy of Sciences}
  \par\nopagebreak
  \textsc{00-656 Warsaw, Poland}  
  \par\nopagebreak
  \textit{E-mail address}: \texttt{mwasilewski@impan.pl}

  % % \medskip
  % % 
  % % \textsc{Department of Mathematics, INSTITUTION}
  % % \par\nopagebreak
  % % \textsc{ADDRESS}
  % % \par\nopagebreak
  % % \textit{E-mail address}: \texttt{??}
  % % 

}}

\maketitle

\begin{abstract}
  A quantum graph $\mathcal{G}$ housed by a matrix algebra $M_n$ can be encoded as an operator system $\mathcal S=\mathcal{S}_{\mathcal{G}}\le M_n$. There are two sensible notions of quantum automorphism group for any such: $\mathrm{Qut}(\mathcal G)$, capturing the quantum symmetries of the adjacency matrix $A:M_n\to M_n$ attached to $\mathcal{G}$, and $\mathrm{Qut}(\mathcal S\le M_n)$, the quantum group acting universally on $M_n$ so as to preserve its $C^*$ structure, standard trace, and subspace $\mathcal{S}\le M_n$. The two quantum groups coincide classically, but diverge in general. We nevertheless show that both are generically trivial in the sense that they are so for $\mathcal{S}\le M_n$ ranging over a non-empty Zariski-open set under all reasonable dimensional constraints on $\dim \mathcal{S}$ and $n$.

  This extends analogous prior results by the first and third authors to the effect that classical symmetry groups of still-quantum graphs are generically trivial, and offers a fully quantum counterpart to the familiar probabilistic almost-rigidity of finite graphs. An auxiliary result sheds some light on the relationship between the two notions of quantum automorphism group, identifying the universal preserver of the quantum adjacency matrix of $\mathcal{G}$ with the quantum automorphism group not of $\mathcal{S}\le M_n$, but rather of the complex conjugate $\overline{\mathcal{S}}\le M_n$. 
\end{abstract}

\noindent {\em Key words:
  Choi matrix;
  Grassmannian;
  Zariski topology;
  comodule;
  compact quantum group;
  operator system;
  quantum automorphism group;
  quantum graph
}

\vspace{.5cm}

\noindent{MSC 2020: 20G42; 46L67; 16T15; 18M05; 46L89; 47L25; 52C25; 05C25

  % 20G42 Quantum groups (quantized function algebras) and their representations
  % 46L67 Quantum groups (operator algebraic aspects)
  % 16T15 Coalgebras and comodules; corings
  % 18M05 Monoidal categories, symmetric monoidal categories
  % 46L89 Other "noncommutative'' mathematics based on $C^*$-algebra theory
  % 47L25 Operator spaces (= matricially normed spaces)
  % 52C25 Rigidity and flexibility of structures (aspects of discrete geometry)
  % 05C25 Graphs and abstract algebra (groups, rings, fields, etc.)
  
}

\tableofcontents

%%%%%%%%%%%%%%%%%%%%%%%%%%%%%%%%
%%%%%%%%%%%%%%%%%%%%%%%%%%%%%%%%
\section*{Introduction}

Quantum graphs have received much recent attention in the literature; the later is by now substantial enough to make it difficult to survey with any pretense at exhaustion in a short recollection such as this, but we mention \cite{zbMATH07202497,MR4140642,zbMATH07856619,zbMATH07632578,MR4481115} as representative sources (with substantive references of their own) along with a number of papers \cite{MR4139106,MR4507619,MR3015725,mrv_qtm-compos,mrv_qtm-qgriso,MR3447998,MR4707042} concerned with quantum graphs' relevance to emergent areas of mathematical physics such as quantum computation and information theory.

By way of a quick recollection:

\begin{definition}\label{def:q.graphs}
  Let $B$ be a finite dimensional $C^{\ast}$-algebra equipped with a faithful tracial functional $\tau$ such that $mm^{\ast} = \id$, where $m: B \otimes B \to B$ is the multiplication map and $m^{\ast}: B \to B\otimes B$ is its adjoint with respect to the inner product  induced by $\tau$, i.e. $\Braket{x\mid y} := \tau(x^{\ast}y)$. We call a completely positive map $A: B\to B$ a \emph{quantum adjacency matrix} if it satisfies $m(A\otimes A)m^{\ast} = A$. The triple $\mathcal{G} := (B, \tau, A)$ will be called a \emph{quantum graph}. 
\end{definition}
\begin{remark}
  Such a functional $\tau$ is unique for $B \simeq \bigoplus_{a=1}^{k} M_{n_{a}}$, namely $\tau(x) = \bigoplus_{a=1}^{k} n_{a} \Tr_{n_{a}}(x_{a})$ for the usual trace $\Tr_n$ on $M_n$. In the non-tracial case such a positive functional is not unique.
\end{remark}

An equivalent definition \cite[Definition 2.6]{zbMATH06008057} is in terms of $B'$-bimodules inside
\begin{equation*}
  \cL(\mathsf{H})
  :=
  \left\{\text{bounded operators on $\mathsf{H}$}\right\}
  ,
\end{equation*}
having represented $B$ faithfully on the Hilbert space $\mathsf{H}$. In the special case $B = M_n$ and $\mathsf{H} = \bC^n$ (which we are concerned with virtually exclusively) these are just subspaces of $M_n$. \emph{Operator systems}, i.e. \cite[pre Proposition 2.1]{pls-bk} unital self-adjoint subspaces correspond to reflexive, undirected quantum graphs. This condition phrased in terms of the quantum adjacency matrix means that $m(A\otimes \id)m^{\ast} = \id$ and $A$ is self-adjoint. An operator system associated with the quantum graph $\cG$ will be denoted by $\cS$.

\begin{definition}\label{def:q.adj.deg.matrix}
  Let $\mathcal{G}$ be a quantum graph. Then $D:=A\mathds{1}$ is called the \emph{degree matrix}. If $\cS$ is the corresponding operator system and $(X_i)_{i}$ is its orthonormal basis with respect to the \emph{normalized} trace $\tr=\frac{\Tr}n$, then the quantum adjacency matrix is given by $Ax = \sum_{i} X_{i} x X_{i}^{\ast}$, so $D = \sum_{i} X_{i} X_{i}^{\ast}$.

  Many quantum graphs in this paper will arise as linear spans $\braket{\cX}$ of operator tuples $\cX$.
\end{definition}

Our primary goal is to establish the generic quantum rigidity of a quantum graph housed by $M_n$, ``fully quantizing''
\begin{itemize}[wide]
\item the analogous \cite[Theorem 3.19]{zbMATH07502493}, to the effect that ``most'' operator systems $\cS\le M_n$, under only the sensible dimensional constraints, have trivial (\emph{classical}) automorphism group;

\item as well as the well-known fully-classical counterpart (\cite[Theorem 3.1]{ksv_randomgr} or \cite[Theorem 9.13]{boll_rnd-grph_2001} for $p=1/2$ modeled on the original \cite[Theorem 2]{MR156334} etc.): selecting each edge of an $n$-vertex graph with probability $0<p<1$, the probability that the graph have trivial symmetry group approaches $1$ as $n\to\infty$.
\end{itemize}

\emph{Quantum rigidity} refers to a quantum graph's having trivial \emph{quantum automorphism group}, which notion requires some unpacking of its own. In first instance, all such will be \emph{compact quantum groups}; the background on the latter needed here is not extensive, and we refer the reader to \cite[Chapter 1]{NeTu13} or \cite[Chapters 3 and 5]{tim} say, with other sources occasionally making an appearance where needed.  

Compact quantum groups that will feature fairly prominently include
\begin{itemize}[wide]
\item the \emph{quantum automorphism groups} $\Qut(B,\tau)$ of finite-dimensional $C^*$-algebras equipped with (typically tracial) states $\tau$ (e.g. \cite[Theorem 6.1(2)]{zbMATH01316946} or \cite[Deﬁnition 1.10]{zbMATH06617237});

\item the \emph{free unitary group} $\U^+(n)$, whose underlying \emph{CQG algebra} \cite[Definition 2.2]{dk_cqg} $H:=\cO\left(\U^+(n)\right)$ is freely generated (as a $*$-algebra) by the entries of a unitary $u:=(u_{ij})_{i,j}\in M_n(H)$ with $\overline{u}:=\left(u^*_{ij}\right)_{i,j}$ again unitary.

  The $C^*$-completion of $\cO\left(\U^+(n)\right)$ is the Hopf $*$-algebra denoted by $A_u(n)$ in \cite[Theorem 1.2]{zbMATH00763766} and $A_u(I_n)$ in \cite[Remark 1.5(2)]{zbMATH00908215}. Note that the $\U^+(n)$-action on $\bC^n$ affords a conjugation action on $(M_n,\tau:=\text{normalized trace})$; we take this for granted repeatedly in the sequel. 
\end{itemize}

We will begin with the more common definition of the quantum automorphism group of a quantum graph, introduced in \cite{MR3849575} and then extended to the non-tracial case in  \cite{bcehpsw_bigal}.
\begin{definition}\label{def:qut.adj}
  Let $\mathcal{G}:= (B, \tau, A)$ be a quantum graph. Let $\Qut(B,\tau)$ be the quantum automorphism group of $(B, \tau)$ and let $\alpha: B \to B \otimes \cO(\Qut(B,\tau))$ be the associated action. We define $\cO(\Qut(\mathcal{G}))$ to be quotient of $\cO(\Qut(B,\tau))$ by the relation $\alpha\circ A = (A\otimes \id)\circ \alpha$ and call the resulting compact quantum group the quantum automorphism group of the quantum graph $\mathcal{G}$.
\end{definition}

Given the description of quantum graphs in terms of operator systems, the following seems natural.

\begin{definition}\label{def:qut.opsys}
  Let $\cS \le B$ be an operator system. We call the quantum subgroup of $\Qut(B, \tau)$ preserving $\cS$ the quantum automorphism group the quantum automorphism group of $\cS \le B$, denoted by $\Qut(\cS \le B)$.
\end{definition}

If $B=M_n$ then $\cS$ can be viewed as a quantum graph on $M_n$, hence we have \emph{two} candidates for the quantum automorphism group. The connection between different notions of the quantum automorphism groups of quantum graphs has been studied by Daws (see \cite[Theorem 9.18]{MR4706978}) and the relationship is not straightforward (see \Cref{ex:cls.quts} below for relatively simple cases of disagreement). Nevertheless, we will show that for our purposes these objects are similar enough: one is generically trivial if and only if the other one is (\Cref{cor:2qgps.gen.triv}).

We refer to $(d+1)$-dimensional operator systems as \emph{operator $d$-systems} (for the unit will not play much of a role in assessing the size of the space of symmetries). Here and throughout the paper a quantum-graph/operator-$d$-system property holds \emph{generically} if it does over a non-empty open set in the real Zariski topology on the Grassmannian $\bG_{sa}(d,M_n)$ of $d$-dimensional self-adjoint complex subspaces of $M_n$. 

In order to state \Cref{th:qut.adj.smlr} we need \emph{spatial} versions of the two quantum groups in \Cref{def:qut.adj,def:qut.opsys}: quantum groups operating on the quantum graph in the appropriate structure-preserving manner, with an action induced by one on $\bC^n$.

\begin{definition}\label{def:spat.qut}
  Let $\G\xrightarrow{\varphi} \U^+(n)$ be a quantum-group morphism inducing
  \begin{equation*}
    \bC^n
    \xrightarrow[\quad\text{$\G$-action}\quad]{\quad}
    \bC^n\otimes \cO(\G),
  \end{equation*}
  hence also a conjugation action
  \begin{equation*}
    \left(
      M_n\cong \bC^n\otimes \left(\bC^n\right)^*\cong \End(\bC^n)
    \right)
    \xrightarrow{\quad\Ad_{\varphi}\quad}
    M_n\otimes \cO(\G)
  \end{equation*}
  with associated morphism
  \begin{equation*}
    \begin{tikzpicture}[>=stealth,auto,baseline=(current  bounding  box.center)]
      \path[anchor=base] 
      (0,0) node (l) {$\G$}
      +(2,.5) node (u) {$\U^+(n)$}
      +(4,0) node (r) {$\Qut(M_n,\tau)$}
      ;
      \draw[->] (l) to[bend left=6] node[pos=.5,auto] {$\scriptstyle \varphi$} (u);
      \draw[->] (u) to[bend left=6] node[pos=.5,auto] {$\scriptstyle $} (r);
      \draw[->] (l) to[bend right=6] node[pos=.5,auto,swap] {$\scriptstyle \varphi_{\Ad}$} (r);
    \end{tikzpicture}
  \end{equation*}
  For quantum graphs $\cG$ on $M_n$ or operator systems $\cS\le M_n$ we use the collective notation $\Qut(\bullet)$ for the objects introduced in \Cref{def:qut.adj,def:qut.opsys}. 
  
  The \emph{$\varphi$- (or $\bG$-)relative quantum automorphism group} of a quantum graph $\cG$ on $M_n$ or an operator system $\cS$ is the object dual to the \emph{pushout} \cite[Definition 2.5.1 and post Proposition 2.5.3]{brcx_hndbk-1}
  \begin{equation*}
    \begin{tikzpicture}[>=stealth,auto,baseline=(current  bounding  box.center)]
      \path[anchor=base] 
      (0,0) node (l) {$\cO\left(\Qut(M_n,\tau)\right)$}
      +(3,.5) node (u) {$\cO\left(\Qut(\bullet)\right)$}
      +(3,-.5) node (d) {$\cO(\G)$}
      +(6,0) node (r) {$\cO\left(\Qut_{\varphi}(\bullet)\right)$}
      ;
      \draw[->>] (l) to[bend left=6] node[pos=.5,auto] {$\scriptstyle $} (u);
      \draw[->] (u) to[bend left=6] node[pos=.5,auto] {$\scriptstyle $} (r);
      \draw[->] (l) to[bend right=6] node[pos=.5,auto,swap] {$\scriptstyle \varphi_{\Ad}$} (d);
      \draw[->>] (d) to[bend right=6] node[pos=.5,auto,swap] {$\scriptstyle $} (r);
    \end{tikzpicture}
  \end{equation*}
  in the category of CQG algebras. 
\end{definition}

\begin{theorem}\label{th:qut.adj.smlr}
  Let $1\leqslant d \leqslant n^2-2$ and $\G\xrightarrow{\varphi} \U^+(n)$ a quantum-group morphism. For an operator $d$-system $\cS\le M_n$ with quantum graph $\mathcal{G}$ we have
  \begin{equation*}
    \Qut_{\varphi}(\mathcal{G})= \Qut_{\varphi}\left(\overline{\cS} \le M_n\right),
  \end{equation*}
  $\overline{\cS}$ denoting the complex conjugate of $\cS$.
\end{theorem}

% % OLD: GENERIC RATHER THAN SPATIAL
% % 
% % \begin{theorem}\label{th:qut.adj.smlr}
% %   Let $1\leqslant d \leqslant n^2-2$. For generic operator $d$-systems $\cS\le M_n$ (and associated quantum graphs $\mathcal{G}$) we have
% %   \begin{equation*}
% %     \Qut(\mathcal{G})= \Qut\left(\overline{\cS} \le M_n\right),
% %   \end{equation*}
% %   $\overline{\cS}$ denoting the complex conjugate of $\cS$.
% % \end{theorem}

That in hand, we can state the announced generic-rigidity result.

\begin{theorem}\label{th:dn23}
  For positive integers $n\ge 3$ and $2\le d\le n^2-3$ the quantum automorphism group $\Qut(\cS\le M_n)$ is trivial for generic operator $d$-system $\cS\le M_n$.
\end{theorem}

\subsection*{Acknowledgements}
MW was supported by the  National Science Center, Poland (NCN) grant no. 2021/43/D/ST1/01446. PS was partially supported by NCN (National Science Centre, Poland) grant no.~2022/47/B/ST1/00582. MW would like to thank Adam Skalski for helpful discussions.

% % OLD: EXPANDED INTO A MORE EXPLICIT VERSION FOR CLARITY 
% % 
% % \begin{theorem}\label{th:dn23}
% %   $\tensor*[_{3\le n}]{\cat{qRig}}{_{2\le d\le n^2-3}}$ holds.
% % \end{theorem}

% % OLD: SNIPPET UNNEEDED
% % 
% % As we will deal with the random quantum graphs introduced in \cite{zbMATH07502493}, it will be useful to briefly recall the notion. The starting point is fixing parameters $n$ and $d$, where $n \ge 3$ is the size of the matrix algebra $M_n$ and $d+1$ is the dimension of the operator system: since the unit will play very little role, we will actually work with the $d$-dimensional orthogonal complement of the unit, i.e. the traceless part. Thus we will refer to $d+1$-dimensional operator systems as $d$-operator systems. The parameter $d$ will lie in the range between $1$ and $n^2-2$ (or $2$ and $n^2-3$).
% %

% % %%%%%%%%%%%%%%%%%%%%%%%%%%%%%%%%
% % \subsection*{Acknowledgments}
% % 

% % %%%%%%%%%%%%%%%%%%%%%%%%%%%%%%%%
% % %%%%%%%%%%%%%%%%%%%%%%%%%%%%%%%%
% % \section{Preliminaries}\label{se:prel}
% %

%%%%%%%%%%%%%%%%%%%%%%%%%%%%%%%%
%%%%%%%%%%%%%%%%%%%%%%%%%%%%%%%%
\section{Two versions of the quantum automorphism group}\label{se:2qgps}

We begin with an analogue of \cite[Proposition 1.13]{zbMATH07502493} saying that all symmetries preserve the degree matrix.

\begin{proposition}\label{prop:degreefixed}
Let $\mathcal{G}:= (M_n, \tau, A)$ be a quantum graph and let $\cS \le M_n$ be the corresponding operator system. Let $D:= A\mathds{1}$ be the degree matrix of $\mathcal{G}$. Then $\alpha(D) = D \otimes \mathds{1}$ if $\alpha$ denotes either the action of $\Qut(\mathcal{G})$ or $\Qut(\cS \le M_n)$ on $M_n$.
\end{proposition}
\begin{proof}
  For $\Qut(\mathcal{G})$ this is clear because $\alpha \circ A = (A\otimes \id)\circ \alpha$, $D= A\mathds{1}$ and $\alpha$ is unital.
 
For $\Qut(\cS \le M_n)$, recall that $D= \sum_{j} e_{j} e_{j}^{\ast}$, where $(e_{j})_{j}$ is an orthonormal basis of $\cS$. Observe that the Hilbert space $\left(M_n,\ \Braket{-\mid -}\right)$ is a \emph{unitary} \cite[(1.8)]{dk_cqg} $H$-comodule for the CQG algebra $H:=\cO(\Qut(M_n,\tau))$, so the same goes for $H'$-subcomodules $\cV\le M_n$ for CQG quotient algebras $H\xrightarrowdbl{} H'$, in particular for $\cS \le M_n$. We have 
  \begin{equation*}
    \cV\ni e_j
    \xmapsto{\quad}
    \sum_i e_i \otimes u_{ij}
    \in \cV\otimes H'
    \quad
    \xRightarrow{\quad}
    \quad
    \cV^*\ni e^*_j
    \xmapsto{\quad}
    \sum_i e^*_i \otimes u^*_{ij}
    \in \cV^*\otimes H',
  \end{equation*}
  identifying \cite[post Definition 1.3.8]{NeTu13} $\cV^*\le M_n$, also $H'$-invariant, with the \emph{dual} comodule in the monoidal-categorical sense of \cite[Definition 2.10.1]{egno}. It follows that
  \begin{equation*}
    \sum_{j}e_j\otimes e^*_j
    \in
    \cV\otimes \cV^*
    \le
    \cM_n^{\otimes 2},
  \end{equation*}
  identifiable with $1\in \End(\cV)\cong \cV\otimes \cV^*$, is $H'$-fixed, and hence so is its image through multiplication (an $H'$-comodule morphism). 
\end{proof}

Similarly to \cite[Theorem 3.24]{zbMATH07502493} one can now conclude that generically the quantum automorphism group of a quantum graph is abelian. Some terms and notation will be useful in discussing this.

\begin{definition}\label{def:qdiag}
  \begin{enumerate}[(1),wide]
  \item\label{item:def:qdiag:diag.gp} The quantum group $\D^+(n)\le \U^+(n)$ is the \emph{Pontryagin dual} \cite[\S 3.3]{tim} $\widehat{F_n}$ of the free group
    \begin{equation*}
      F_n\cong \Braket{s_i,\ 1\le i\le n}
      ,\quad
      s_i\text{ free generators},
    \end{equation*}
    coacting on $\bC^n$ in the obvious fashion: the standard basis $(e_i)_{i=1}^n\subset \bC^n$ is homogeneous for an $F_n$-grading assigning $e_i$ degree $s_i$.   
    
  \item\label{item:def:qdiag:diag.act} $\D^+(n)$ acts on $\bC^n$ as well as $(M_n,\tau)$, by conjugation. We will say that a quantum group $\G$ mapping to either $\U^+(n)$ or $\Qut(M_n,\tau)$ acts
    \begin{itemize}[wide]
    \item \emph{diagonally} on $\bC^n$ or $M_n$ if $\G\to \U^+(n)$ factors through $\D^+(n)$, and similarly for the action on $M_n$;

    \item \emph{essentially diagonally} if the same holds up to a change of orthonormal basis for $\bC^n$. 
    \end{itemize}
  \end{enumerate}  
\end{definition}

\begin{proposition}\label{pr:leave.diag.inv}
  If $X\in M_n$ is such that $X^*X$ has simple spectrum then the action of $\Qut(\bC X\subseteq M_n)$ on $M_n$ is essentially diagonal in the sense of \Cref{def:qdiag}\Cref{item:def:qdiag:diag.act}.
\end{proposition}
\begin{proof}
  The coaction of $\cO(\G)$, $\G:=\Qut(\bC X\subseteq M_n)$ is by assumption of the form
  \begin{equation*}
    X
    \xmapsto{\quad\rho\quad}
    X\otimes \delta
    ,\quad
    \delta\in \cO(\G)
    \text{ \emph{group-like} \cite[Deﬁnition 2.1.10]{rad}}.
  \end{equation*}
  $\rho$ thus fixes the simple-spectrum self-adjoint operator $X^* X$; assuming that operator diagonal, the coaction will factor through $M_n\to M_n\otimes \cO(\D^+(n))$. 
\end{proof}

\begin{corollary}\label{cor:gen.diag}
  For $3\le n$ and $1\le d\le n^2-2$ $\Qut(\cS\le M_n)$ and $\Qut(\mathcal{G})$ generically act essentially diagonally in the sense of \Cref{def:qdiag}\Cref{item:def:qdiag:diag.act}.  
\end{corollary}
\begin{proof}
  In both cases the degree matrix $D$ is fixed with almost-surely simple spectrum \cite[proof of Theorem 3.24]{zbMATH07502493}, hence the conclusion via \Cref{pr:leave.diag.inv}.
\end{proof}

This in hand, we can revisit the issue of contrasting the two notions of quantum automorphism group discussed in \Cref{def:qut.adj,def:qut.opsys}. 

\begin{example}\label{ex:cls.quts}
  By \Cref{cor:gen.diag} it suffices to assume that $\Qut(\Braket{\cX}\le M_n)$ operates via the conjugation of a $\Gamma$-grading, where $F_n\xrightarrowdbl{}\Gamma$ is a quotient of the free group $F_n=\Braket{s_i}_{i=1}^n$ assigning the standard basis element $e_i\in \bC^n$ degree $s_i$. The claim is that even when $Y_j$ are respectively $\gamma_j$-homogeneous for $\gamma_j\in \Gamma$ (which we can always assume under the circumstances),
  \begin{equation*}
    M_n \ni x
    \xmapsto{\quad \Phi_{\cS}\in \End(M_n)\quad}
    \sum_{j=1}^{\dim \cS} Y_j x Y_j^*
    \in M_n
    ,\quad
    \cS:=\bC\oplus \Braket{\cX}
  \end{equation*}
  need not be a graded map.
  
  Indeed, consider a \emph{single} $\gamma$-homogeneous $Y=Y_j$ for non-central $\gamma\in \Gamma$ (so that $\Braket{\cX}=\spn\left\{Y,Y^*\right\}$) in such a fashion as to ensure the existence of some $\gamma'$-homogeneous $x\in M_n$ with $[\gamma,\gamma']\ne 1$ and $YxY^*\ne 0$. This is easily arranged: set $n:=|\Gamma|$ for a finite group $\Gamma$, identify $\bC^n\cong \bC\Gamma$ (the group algebra), and take for $Y$ and $x$ the left translation operators by non-commuting $\gamma,\gamma'\in \Gamma$ respectively. We then have
  \begin{equation*}
    YxY^*=\text{translation by $\gamma\gamma'\gamma^{-1}\ne \gamma'$},
  \end{equation*}
  meaning that $Y\cdot Y^*$ does not preserve degrees. 
\end{example}

Classically the automorphism group of $M_n$ is the \emph{projective} unitary group and one cannot go through the usual unitary group, in particular the automorphism group does not act on $\bC^n$. In our setting, generically the quantum automorphism group will sit inside $\D^+(n)$, so it acts on $\bC^n$. This nice algebraic property will allow us to compare the two versions of the quantum automorphism group.

\begin{proposition}\label{prop:equiv.choi}
 Let $V$ and $W$ be finite-dimensional unitary representations of a compact quantum group. Then the Choi correspondence between completely positive maps $\Phi: V\otimes V^{\ast} \to W \otimes W^{\ast}$ and positive operators $C_{\Phi} \in \End(W^{\ast}\otimes V)$ is equivariant.
\end{proposition}
\begin{proof}
This is essentially \cite[Theorem 4.13]{MR4482713}. We provide some details for the reader's convenience. Note that
  \begin{equation}\label{eq:def.cphi}
    \begin{tikzpicture}[>=stealth,auto,baseline=(current  bounding  box.center)]
      \path[anchor=base] 
      (0,0) node (l) {$W^*\otimes V$}
      +(-1,2) node (ul) {$W^*\otimes V\otimes V^*\otimes V$}
      +(7,2) node (ur) {$W^*\otimes W\otimes W^*\otimes V$}
      +(6,0) node (r) {$W^*\otimes V$}
      ;
      \draw[->] (l) to[bend left=6] node[pos=.5,auto] {$\scriptstyle \id_{W^*\otimes V}\otimes \db_{V^*}$} (ul);
      \draw[->] (ul) to[bend left=6] node[pos=.5,auto] {$\scriptstyle \id_{W^*}\otimes \Phi\otimes \id_V$} (ur);
      \draw[->] (ur) to[bend left=6] node[pos=.5,auto] {$\scriptstyle \ev_W\otimes \id_{W^*\otimes V}$} (r);
      \draw[->] (l) to[bend right=6] node[pos=.5,auto,swap] {$\scriptstyle C_{\Phi}$} (r);
    \end{tikzpicture}
  \end{equation}
  where
  \begin{itemize}[wide]
  \item asterisks denote duals in \emph{rigid monoidal $C^*$-categories} \cite[Definition 2.2.1]{NeTu13};

  \item and the morphisms
    \begin{equation*}
      \mathbf{1}
      \xrightarrow[\quad\text{for `dual basis'}\quad]{\quad\db_Z\quad}
      Z\otimes Z^*
      \quad\text{and}\quad
      Z^*\otimes Z
      \xrightarrow[\quad\text{for `evaluation'}\quad]{\quad\ev_Z\quad}
      \mathbf{1}
    \end{equation*}
    ($\mathbf{1}$ being the monoidal unit) are those witnessing the duality (what \cite[Definition 2.2.1]{NeTu13} would denote by $\overline{R}$ and $R^*$ respectively).
  \end{itemize}
This means if $\Phi$ is equivariant, i.e. is a morphism in the representation category, then so is $C_{\Phi}$. A similar formula expresses $\Phi$ in terms of $C_{\Phi}$.  
\end{proof}

The following result is worked out here for completeness; it is easily extracted from \cite[Remark 4 and proof of Theorem 1]{MR376726}, modulo inessential linguistic differences. It relates the operator system to the Choi matrix of the quantum adjacency matrix equivariantly.

\begin{lemma}\label{le:trnsp.inv}
  Let $V$ and $W$ be finite-dimensional unitary representations of a compact quantum group and
  \begin{equation}\label{eq:gen.cp}
    V\otimes V^*
    \xrightarrow{\quad\Phi=\sum_{\ell} Y_{\ell}\cdot Y^*_{\ell}\quad}
    W\otimes W^*
    ,\quad
    Y_{\ell}\in W\otimes V^*
  \end{equation}
  an equivariant completely positive map. 

  The image of the attached equivariant morphism $C_{\Phi}$ of \Cref{eq:def.cphi} is precisely the span of $\overline{Y_{\ell}}$, using the identification 
  \begin{equation}\label{eq:trnsp}
  W^{\ast}\otimes V \simeq \overline{W} \otimes V \simeq \overline{W \otimes \overline{V}}.
  \end{equation}
 
 % \begin{equation}\label{eq:trnsp}
 %   \Hom(W,V)
 %   \cong
 
%   V\otimes W^*
%    \xrightarrow{\quad(\bullet)^t\quad}
%    W^*\otimes V
%    \cong
%    \Hom(V^*,W^*)
%  \end{equation}
%  denoting the flip map, i.e. matrix transposition.
\end{lemma}
\begin{proof}
  We begin by examining the map $C_{\Phi}$ (frequently referred to as the \emph{Choi matrix} \cite[proof of Theorem 3]{MR3114209} of $\Phi$, in light of \cite[Theorem 2]{MR376726}) in the specific case of $\Phi=Y\cdot Y^*$ for $Y\in \Hom(V,W)\cong W\otimes V^*$. Some pertinent notation:
\begin{itemize}[wide]
\item $e$s and $f$s denote basis elements of $V$, $V^*$ and $W$, $W^*$ respectively;

\item lower (upper) induces indicate basis elements for $V$, $W$ and their dual counterparts in $V$, $W^*$ respectively;

\item repeated indices in the same expression indicate a suppressed sum (the familiar \cite{mw_einstein-summation} \emph{Einstein summation} convention). 
\end{itemize}
All of this in place, tracing a basis element $f^r\otimes e_i\in W^*\otimes V$ through \Cref{eq:def.cphi} produces
\begin{equation*}
  \begin{tikzpicture}[>=stealth,auto,baseline=(current  bounding  box.center)]
    \path[anchor=base] 
    (0,0) node (1) {$f^r\otimes e_i$}
    +(1,1) node (2) {$f^r\otimes e_i\otimes e^j\otimes e_j$}
    +(7,1) node (3) {$f^r\otimes \Phi(e_{ij})\otimes e_j$}
    +(8,0) node (4) {$f^r\circ \Phi(e_{ij})\otimes e_j$}
    +(0,-1) node (5) {$f^r\circ (y_{pi}\overline{y_{qj}})_{p,q}\otimes e_j$}
    +(3,-2) node (6) {$(y_{ri}\overline{y}_{qj})_{q}\otimes e_j$}
    +(6,-2) node (7) {$y_{ri}\overline{y}_{qj}f^{q}\otimes e_j$}
    +(9,-1.5) node (8) {$y_{ri}\overline{Y}$}
    ;

    \draw[|->] (1) to[bend left=6] node[pos=.5,auto] {$\scriptstyle $} (2);
    \draw[|->] (2) to[bend left=6] node[pos=.5,auto] {$\scriptstyle $} (3);
    \draw[|->] (3) to[bend left=6] node[pos=.5,auto] {$\scriptstyle $} (4);
    %\draw[double equal sign distance] (4) to[bend left=6] node[pos=.5,auto] {$\scriptstyle $} (5);
    \draw[double equal sign distance] (4) .. controls +(5,-1) and +(-5,1) .. (5) node[pos=.5,auto,swap] {$\scriptstyle \Phi=Y\cdot Y^*$};
    \draw[double equal sign distance] (5) to[bend right=6] node[pos=.5,auto] {$\scriptstyle $} (6);
    \draw[double equal sign distance] (6) to[bend right=6] node[pos=.5,auto] {$\scriptstyle $} (7);
    \draw[double equal sign distance] (7) to[bend right=6] node[pos=.5,auto] {$\scriptstyle $} (8);
  \end{tikzpicture}
\end{equation*}
In the general case of \Cref{eq:gen.cp} the range of $C_{\Phi}$ will thus be precisely the span of $\sum_{\ell}y_{\ell,ri}\overline{Y}_{\ell}$ for varying $i$ and $r$. That
\begin{equation*}
  \spn\left\{\sum_{\ell}y_{\ell,ri}\overline{Y}_{\ell}\right\}_{i,r}
  =
  \spn\left\{\overline{Y}_{\ell}\right\}_{\ell}
\end{equation*}
is easily seen: the $\le$ inclusion is self-evident, and the dimensions coincide (with the \emph{rank} \cite[Definition 3.35]{hck_tens_2e_2019} of the tensor $(y_{\ell,ri})_{\ell,i,r}$). 
\end{proof}

\pf{th:qut.adj.smlr}
\begin{th:qut.adj.smlr}
   Let $\alpha$ be the action of $\Qut_{\varphi}(\mathcal{G})$ on $M_n$. From \Cref{prop:equiv.choi} we know that the Choi matrix of the quantum adjacency matrix of $\mathcal{G}$ is equivariant. It follows that its range, identified with $\overline{\cS}$ in \Cref{le:trnsp.inv}, is invariant and hence $\Qut_{\varphi}(\mathcal{G})\le \Qut_{\varphi}(\overline{\cS} \le M_n)$. On the other hand, the action $\beta$ of $\Qut_{\varphi}(\overline{\cS} \le M_n)$ on $M_n$ leaves $\overline{\cS}$ invariant. The Choi matrix must thus be equivariant, and with it the associated quantum adjacency matrix.
\end{th:qut.adj.smlr}

Naturally, being interested in generic behavior, knowledge on almost all complex conjugates $\overline{\cS}$ translates to knowledge on almost all operator systems $\cS$.

\begin{corollary}\label{cor:diag.spat}
  Let $1\leqslant d \leqslant n^2-2$. For generic operator $d$-systems $\cS\le M_n$ (and associated quantum graphs $\mathcal{G}$) we have
  \begin{equation*}
    \Qut(\mathcal{G})= \Qut\left(\overline{\cS} \le M_n\right),
  \end{equation*}
  $\overline{\cS}$ denoting the complex conjugate of $\cS$.
\end{corollary}
\begin{proof}
  It follows from \Cref{cor:gen.diag} that both actions on $M_n$ are essentially diagonal, so \Cref{th:qut.adj.smlr} applies to the embedding $\D^+(n)\lhook\joinrel\xrightarrow{\varphi}\U^+(n)$. 
\end{proof}

In particular:

\begin{corollary}\label{cor:2qgps.gen.triv}
  $\Qut(\cS \le M_n)$ is generically trivial if and only if $\Qut(\mathcal{G})$ is.  \qedhere
\end{corollary}

% % OLD: GENERIC, NOT SPATIAL
% % 
% % \pf{th:qut.adj.smlr}
% % \begin{th:qut.adj.smlr}
% %   It follows from \Cref{cor:gen.diag} that both actions on $M_n$ are essentially diagonal, therefore both quantum groups act on $\bC^n$ in such a way that $M_n \simeq \bC^n \otimes (\bC^n)^{\ast}$. Let $\alpha$ be the action of $\Qut(\mathcal{G})$ on $M_n$. From \Cref{prop:equiv.choi} we know that the Choi matrix of the quantum adjacency matrix of $\mathcal{G}$ is equivariant. It follows that its range, identified with $\overline{\cS}$ in \Cref{le:trnsp.inv}, is invariant and hence $\Qut(\mathcal{G})\le \Qut(\overline{\cS} \le M_n)$. On the other hand, the action $\beta$ of $\Qut(\overline{\cS} \le M_n)$ on $M_n$ leaves $\overline{\cS}$ invariant. The Choi matrix must thus be equivariant, and with it the associated quantum adjacency matrix.
% % \end{th:qut.adj.smlr}

\begin{remark}\label{re:redo.gp.grd}
  It will be instructive to revisit \Cref{ex:cls.quts} and illustrate how the complex conjugation \Cref{eq:trnsp} (related to the opposite action of \cite[Theorem 9.18]{MR4706978}) affects matters in one particular class of examples. 

In the present case we have $V=W=\bC \Gamma$, and the index set $I\ni i$ for a basis $(e_i)_I$ is $\Gamma$ itself. For $Y$ we choose
  \begin{equation*}
    Y=\sum_{i\in \Gamma} e_{\gamma i,i}
    \quad\text{with}\quad
    \deg e_{pq}=pq^{-1}\in \Gamma.
  \end{equation*}
  Now, the complex conjugate $\overline{Y}$ will be $\sum_i f_{\gamma i, i}$ ($e\to f$ change in lettering will help with the bookkeeping), with the caveat that these are now matrix units in $V^*\otimes V\cong \End(V^*)$ instead (cf. \Cref{eq:trnsp}) and hence
  \begin{equation*}
    \deg f_{pq}=p^{-1}q
    \xRightarrow{\quad}
    \deg f_{\gamma i, i}=i^{-1}\gamma^{-1} i
    \xRightarrow{\quad}
    \overline{Y}\in \End\left(V^*\right)
    \text{ is not (generally) homogeneous}.
  \end{equation*}
  We thus have
  \begin{equation*}
    (\overline{Y})\bullet (\overline{Y})^{\ast}
    =
    \left(\sum_j f_{\gamma j, j} \right)\bullet \left(\sum_i f_{i,\gamma i}\right)
    \quad
    \text{acting on}
    \quad
    \End(V^*),
  \end{equation*}
  which \emph{is} a graded map: it sends the degree-$j^{-1}i$-element $f_{ji}$ to the element $f_{\gamma j,\gamma i}$, of degree
  \begin{equation*}
    (\gamma j)^{-1}\cdot \gamma i = j^{-1}i.
  \end{equation*}
  In short:
  \begin{itemize}[wide]
  \item $Y\in \End(V)$ is homogeneous, but its associated map $Y\cdot Y^{\ast}$ on $\End(V)$ is not;

  \item on the other hand, the map $\overline{Y}\cdot (\overline{Y})^{\ast}$ attached to the \emph{in}homogeneous $\overline{Y}\in \End(V^*)$ \emph{is} homogeneous on $\End(V^*)$.
  \end{itemize}
\end{remark}

%%%%%%%%%%%%%%%%%%%%%%%%%%%%%%%%
%%%%%%%%%%%%%%%%%%%%%%%%%%%%%%%%
\section{Generic operator systems}\label{se:gen.osys}

In this section we prove that generic quantum graphs do not have quantum symmetries. It follows from \Cref{se:2qgps} that it does not matter, which notion of the quantum automorphism group we use for this particular task, so we can choose one; it will be more convenient for us to use $\Qut(\cS \le M_n)$.

\begin{definition}\label{def:gen.rig}
  An operator system $\cS\le B$ in a finite-dimensional $C^*$-algebra $B$ equipped with a state $\tau$ is \emph{$(B,\tau)$-quantum-rigid} (just \emph{quantum-rigid} when the setting is understood) if the quantum automorphism group $\Qut(\cS\le B)$ operates on $\cS$ trivially.

  $B$ will mostly be $M_n$, with $\tau$ the usual tracial state unless the convention is explicitly overruled. 
\end{definition}

\begin{definition}\label{def:nd.qrig}
  Let $n\in \bZ_{\ge 1}$ and $1\le d\le n^2$.

  \begin{enumerate}[(1),wide]
  \item We refer to the claim that for the generic operator $d$-system $\cS\le M_n$ the quantum group $\Qut(\cS\le M_n)$ is trivial as $\tensor*[_{n}]{\cat{qRig}}{_{d}}$.

  \item We also write $\tensor*[_{P}]{\cat{qRig}}{_{Q}}$ for the conjunction of $\tensor*[_{n}]{\cat{qRig}}{_{d}}$ with the parameters $n$ and $d$ constrained by various conditions $P$ and $Q$. The paradigmatic example is $\cat{qRig}_{3\le n\le N,[\alpha(n),\beta(n)]}$, meaning
    \begin{equation*}
      \forall\left(3\le n\le N\right)
      \forall\left(d\in [\alpha(n),\beta(n)]\right)
      \quad:\quad
      \tensor*[_{n}]{\cat{qRig}}{_{d}}
    \end{equation*}
    (for functions $\alpha(n)\le \beta(n)$).

  \item The $\cat{qRig}$ statements also specialize to individual tuples of matrices (or subspaces of $M_n$): $\cat{qRig}(\cX)=\cat{qRig}(\Braket{\cX})$ for a collection $\cX\subset M_n$ means that the quantum automorphism group $\Qut(\braket{\cX}\subset M_{n})$ is trivial.

  \item It will be convenient, on occasion, to relativize the $\cat{qRig}$ conditions: we say that one such holds \emph{in (or relative to) $\G$} for a quantum subgroup $\G\le \Qut(M_n,\tau)$ if the quantum subgroup of $\G$ preserving $\Braket{X}$ is trivial for generic $\cX$. 
  \end{enumerate}
\end{definition}

One of the important tools for extending results from \cite{zbMATH07502493} will be the following.

\begin{theorem}\cite[Theorem A]{2505.07485v1}\label{thm:gen.triv}
Let $V$ be a finite dimensional representation of a compact quantum group $\G$ and let $\G_{W}\le \mathbb{G}$ be the isotropy quantum subgroup of $W \le \bG(d,V)$, where $\bG(d,V)$ is the complex Grassmannian consisting of $d$-dimensional subspaces of $V$. Then the subset $\{W \in \bG(d,V): \G_W \text{ acts trivially on }W\text{ or }V \}$ is open in the Zariski topology on the \emph{Weil-restricted} \cite[\S 7.6]{blr_neron} $\mathrm{Res}_{\bC\slash \mathbb{R}}(\bG(d,V))$, hence of full measure if non-empty.
\end{theorem}

Before turning to the proof of our main result (\Cref{th:dn23}), we outline its structure. The main inductive step will be presented in \Cref{pr:ind.stp}. Using it and \Cref{th:dn1n2} will reduce proving to a number of small cases, which will be covered in \Cref{subsec:small.case} (see  \Cref{sec:comp.alg} for an alternative approach using computer verification).

We will frequently need the following lemma (analogue of \cite[Lemma 3.11]{zbMATH07502493}).
\begin{lemma}\label{lem:orth.compl}
Let $\cS \le M_n$ be an operator system and let $\cS^{\perp}$ be its \emph{reflexive complement}, we take take its traceless part (i.e. the orthogonal complement of the unit), then the orthogonal complement inside the space of traceless matrices, and finally add the unit. Then $\Qut(\cS \le M_n) \simeq \Qut(\cS^{\perp} \le M_n)$. In particular $\tensor*[_n]{\cat{qRig}}{_d} \leftrightarrow \tensor*[_n]{\cat{qRig}}{_{n^2-d-1}}$.
\end{lemma}
\begin{proof}
The action of the quantum automorphism group preserves the unit and the inner product, hence also orthogonal complement. The conclusion easily follows.
\end{proof}

We will now present the following proposition, in whose proof we show, how the main theorem follows from small cases and it motivates the results we state next.

\begin{proposition}\label{pr:big.cond}
  We have
  \begin{equation*}
    \tensor*[_n]{\cat{qRig}}{_d}
    \ \text{for}\
    (n,d)\in \left\{(3,3),\ (3,4),\ (4,7)\right\}
    \quad
    \xRightarrow{\quad}
    \quad
    \tensor*[_{3\le n}]{\cat{qRig}}{_{2\le d\le n^2-3}}.
  \end{equation*}
\end{proposition}
\begin{proof}
  \Cref{th:dn1n2} ensures that $\tensor*[_3]{\cat{qRig}}{_{2}}$ and $\tensor*[_4]{\cat{qRig}}{_{\{2\ldots 6\}}}$ hold, and the hypothesis together with the symmetry
  \begin{equation*}
    \tensor*[_n]{\cat{qRig}}{_d}
    \quad
    \xLeftrightarrow{\quad}
    \quad
    \tensor*[_n]{\cat{qRig}}{_{n^2-1-d}}
  \end{equation*}
  completes the range for $\tensor*[_{n\in \left\{3,4\right\}}]{\cat{qRig}}{_{2\le d\le n^2-3}}$. For $n\ge 5$ we have $(n-1)(n-2)\ge 2n$, so the conclusion follows from \Cref{cor:2n.enough} and \Cref{th:dn1n2}. 
\end{proof}

We can now state and prove an analogue of \cite[Proposition 3.13]{zbMATH07502493}, which is the main inductive step.

\begin{proposition}\label{pr:ind.stp}
  Let $n\in \bZ_{\ge 3}$ and assume generic operator $d$-systems in $M_n$ are quantum-rigid for $2\le d\le n^2-3$. The same holds, then, of generic operator $d'$-systems in $M_{n+1}$ for $2\le d'\le n^2-2$. 
\end{proposition}
\begin{proof}
  This will be a simple matter of retracing the steps proving the aforementioned \cite[Proposition 3.13]{zbMATH07502493}, employing \Cref{thm:gen.triv} in place of \cite[Lemma 3.12]{zbMATH07502493}. 
  
  A generic $d$-tuple $\cX$ of traceless self-adjoint operators in $M_n\subset M_{n+1}$ (upper left-hand corner embedding) will generate the corner (non-unital) subalgebra $M_n\subset M_{n+1}$, so the hypothesis ensures that $\Qut(\braket{\cX}\subseteq M_{n+1})$ operates on $\braket{\cX}$ trivially. The same holds for a (real-)Zariski-open set of tuples by \cite[Theorem A]{2505.07485v1}; the conclusion follows from the fact that generic tuples in $M_{n+1}$ generate the latter as an algebra,
  
  This handles all cases $2\le d'\le n^2-3$. For the remaining value $d'=n^2-2$, retrace the previous argument with a tuple $\cX'$ obtained from $\cX\subset M_n$ by appending a diagonal operator with non-zero $(n+1)\times (n+1)$ entry. Generically such gadgets will generate a block-diagonally-embedded subalgebra $M_n\times \bC\subset M_{n+1}$, again invariant under $\Qut(\braket{\cX}\subset M_{n+1})$. 

  It follows from \cite[Proposition 5.4]{zbMATH06617237} that the individual blocks $M_n$, $n\ge 2$ and $\bC$ of $M_n\times \bC$ are left invariant by the quantum automorphism group of the tracial state inherited from the embedding $M_n\times \bC\subset M_{n+1}$, so the argument does still apply. 
\end{proof}

\begin{corollary}\label{cor:2n.enough}
  For $\bZ_{\ge 3} \ni M\le N$ we have
  \begin{equation*}
    \tensor*[_{M}]{\cat{qRig}}{_{[2,M^2-3]}}
    \ \wedge\ 
    \tensor*[_{M\le n\le N}]{\cat{qRig}}{_{[2,2n]}}
    \xRightarrow{\quad}
    \tensor*[_{M\le n\le N}]{\cat{qRig}}{_{[2,n^2-3]}}.
  \end{equation*}  
\end{corollary}
\begin{proof}

  % % For $n=3$ there is nothing to prove, for in that case $n^2-3=2n=6$. Beyond that we use induction, whose step will effect the passage from $n\le N-1$ to $n+1$.
  % % 
  
  We use induction, whose step will effect the passage from $n\le N-1$ to $n+1$. In first instance, \Cref{pr:ind.stp} provides $\tensor*[_{n+1}]{\cat{qRig}}{_{[2,n^2-2]}}$. On the other hand, because $\tensor*[_{n}]{\cat{qRig}}{_{d}}$ is equivalent to $\tensor*[_{n}]{\cat{qRig}}{_{n^2-1-d}}$, the missing values $d\in [n^2-1,(n+1)^2-3]$ can be recovered from $d\in [2,2n+1]=[2,2(n+1)-1]$, provided by the hypothesis. 
\end{proof}

\begin{remark}
  The proof in fact delivers slightly more than the statement claims: one only needs the bounds $2\le d\le 2n$ for $n=3$; for $n\ge 4$ it suffices to cover $2\le d\le 2n-1$. 
\end{remark}

Other portions of \cite{zbMATH07502493} replicate in the quantum setting. An example is the following counterpart to \cite[Corollary 3.17]{zbMATH07502493}, covering part of the range of interest for $d$.

\begin{theorem}\label{th:dn1n2}
  $\tensor*[_{3\le n}]{\cat{qRig}}{_{2\le d\le (n-1)(n-2)}}$ holds. 
\end{theorem}
\begin{proof}
  The argument precisely parallels the one proving \cite[Corollary 3.17]{zbMATH07502493}, relying on the appropriate analogues to \cite[Lemmas 3.14 and 3.15]{zbMATH07502493}: \Cref{pr:no.non.scl} and \Cref{le:real.inv} respectively. 
\end{proof}

\cite[Lemma 3.14]{zbMATH07502493} refers to diagonal unitaries, recast in the present context in \Cref{def:qdiag}. The quantum version of that result is as follows.

\begin{proposition}\label{pr:no.non.scl}
  Suppose that for some $n \ge 3$ and $1 \le d$ there is a self-adjoint $d$-tuple $\cY\subset M_{n-1}$ of zero-diagonal matrices so that $\cat{qRig}(\cY)$ holds relative to $\D^+(n-1)$. We then have $\tensor*[_n]{\cat{qRig}}{_{d+1}}$.
\end{proposition}
\begin{proof}
  We unwind the proof of \cite[Lemma 3.14]{zbMATH07502493}, paying additional attention to the relevant quantum features. 
  
  First, $\cY$ can be chosen generically among self-adjoint zero-diagonal tuples, by \cite[Theorem A]{2505.07485v1}. Extend $\cY\subset M_{n-1}< M_n$ (upper left-hand corner) generically to $\cX:=\cY\cup X$ for a diagonal, self-adjoint, traceless $X$. $\cX$ can be assumed to generate $M_{n-1}\times \bC<M_n$, whose factors are left invariant by $\G:=\Qut(\Braket{X}\subset M_n)$ by another application of \cite[Proposition 5.4]{zbMATH06617237}.

  Denote by primes the upper left-hand blocks of matrices:
  \begin{equation*}
    M_n
    \ni Z
    \xmapsto{\quad}
    Z'
    \in M_{n-1}<M_n.
  \end{equation*}
  I next claim that $\G$ fixes $X'$: indeed, $\G$ leaves invariant the trace and hence also the affine space
  \begin{equation}\label{eq:z.tr.x'}
    \left\{Z\in \Braket{\cX'}\ :\ \tau(Z)=\tau(X')\right\}.
  \end{equation}
  Simply observe, then, that $X'$ is the shortest element of \Cref{eq:z.tr.x'} with respect to the \emph{Hilbert-Schmidt norm} \cite[(5.6.0.2)]{hj_mtrx} (again preserved by $\G$).

  Fixing the generic (traceless, self-adjoint) diagonal operator $X$, $\G$ must operate via the conjugation action of $\D^+(n)$. It thus fixes $\Braket{\cY}$ pointwise by assumption, and hence also $\Braket{\cX}$. The conclusion follows. 
\end{proof}

\cite[Lemma 3.15]{zbMATH07502493} is concerned with \emph{real} linear subspaces of $\bC^p$ invariant under diagonal unitaries in $\U(p)$. The very notion requires a bit of unpacking in the present quantum setup.

\begin{definition}\label{def:real.inv.qgp}
  Let
  \begin{equation*}
    V
    \xrightarrow{\quad\rho\quad}
    V\otimes_{\bC}\cO(\G)    
  \end{equation*}
  be a representation of a compact quantum group $\G$. A real subspace $W\le V$ is \emph{$\G$- (or $\cO(\G)$-)invariant} if the coaction on $V$ factors as
  \begin{equation*}
    \begin{tikzpicture}[>=stealth,auto,baseline=(current  bounding  box.center)]
      \path[anchor=base] 
      (0,0) node (l) {$W$}
      +(3,.5) node (u) {$V$}
      +(3,-.5) node (d) {$W\otimes_{\bR}\cO(\G)_{sa}$}
      +(6,0) node (r) {$V\otimes_{\bC}\cO(\bG)$}
      ;
      \draw[right hook->] (l) to[bend left=6] node[pos=.5,auto] {$\scriptstyle $} (u);
      \draw[->] (u) to[bend left=6] node[pos=.5,auto] {$\scriptstyle \rho$} (r);
      \draw[->] (l) to[bend right=6] node[pos=.5,auto,swap] {$\scriptstyle \rho_W$} (d);
      \draw[right hook->] (d) to[bend right=6] node[pos=.5,auto,swap] {$\scriptstyle \text{obvious inclusion}$} (r);
    \end{tikzpicture}
  \end{equation*}
  for a coaction $\rho_W$ by the (real) coalgebra $\cO(\G)_{sa}\le \cO(\G)$ of self-adjoint elements. 
\end{definition}

\begin{lemma}\label{le:real.inv}
  Let $2\le p$ and $1\le m\le 2p-1$ be two positive integers. For a generic $m$-dimensional subspace $W\le \bC^p$ the only quantum subgroups of $\D^+(p)$ leaving $W$ invariant in the sense of \Cref{def:real.inv.qgp} are those contained in the classical group
  \begin{equation*}
    \left\{\pm 1\right\}\cong \bZ/2\le \D(p)\le \D^+(p).
  \end{equation*}
\end{lemma}
\begin{proof}
  We retain the notation of \Cref{def:qdiag}, realizing $\D^+(p)$ as $\widehat{F_p}$ for the free group on $p$ generators $s_i$.

  A quantum subgroup $\G\le \D^+(p)$ will be Pontryagin dual to a quotient $F_p\xrightarrowdbl{} \Gamma$, with the $\G$-action corresponding to a $\Gamma$-grading on $\bC^p$. If at least two generators $s_i$, say $s_1$ and $s_2$ have distinct images in $\Gamma$, we have a splitting
  \begin{equation*}
    \bC^p = V_0\oplus V_1,
      \end{equation*}
  where $V_{0}:= \textrm{span}\{e_{k}: e_{k}\text{ has degree }s_1\}$ and $V_1$ is its orthogonal complement.
 If $W$ is $\G$-invariant, it must be compatible with this splitting, that is
  \begin{equation*}
    W=\left(W\cap V_0\right)\oplus \left(W\cap V_1\right).
  \end{equation*}
  Writing $d_j:=\dim_{\bC}V_j$, $j=0,1$ (so that $d_0 + d_1=p$) and $\bG_{\bR}$ for real Grassmannians, observe that
  \begin{equation*}
    \begin{aligned}
      m(2p-m)=\dim \bG_{\bR}(m,\bC^p)
      &>
        \max_{\substack{V_0,V_1\\0\le \ell_j\le 2d_j\\\ell_0+\ell_1=m}}
      \dim \left(\bG(\ell_0,V_0)\times \bG(\ell_1,V_1)\right)\\
      &=
        \max_{\substack{V_0,V_1\\0\le \ell_j\le 2d_j\\\ell_0+\ell_1=m}}
      \left(\sum_{j=0,1}\ell_j(2d_j-\ell_j)\right).
    \end{aligned}    
  \end{equation*}
  Generic $W$, then, can only be invariant under subgroups $\G$ of the classical group of scalars $\bS^1\le \D^+(p)$ dual to the surjection $F_p\xrightarrowdbl{} \bZ$ identifying all $s_i$ with a generator of $\bZ$.

  We can now switch to classical language: generic $W\le \bC^{p}$ will fail to be compatible with the complex structure (i.e. invariant under complex multiplication), so cannot be invariant under scalars other than $\pm 1$.
\end{proof}

The discussion thus far suffices to whittle down the problem to a small number of cases: precisely those addressed classically by \cite[Proposition 3.21]{zbMATH07502493}.

We can now turn to the definitive version of \Cref{th:dn1n2}, covering the entire expected range.

\pf{th:dn23}
\begin{th:dn23}
  We only have to cover the three cases listed in \Cref{pr:big.cond}:
  \begin{itemize}[wide]
  \item the cases $n=3$, $d\in \left\{3,4\right\}$ are settled by \Cref{pr:nn} and \Cref{le:nd34} respectively;

  \item while \Cref{pr:n122} handles $\tensor*[_{4\le n}]{\cat{qRig}}{_{(n-1)^2-2}}$ (and hence $(4,7)$).  \qedhere
  \end{itemize}
\end{th:dn23}

\subsection{Small cases}\label{subsec:small.case}

\begin{proposition}\label{pr:nn}
  $\tensor*[_{3\le n}]{\cat{qRig}}{_n}$ holds.
\end{proposition}
\begin{proof}
  Set $\G:=\Qut(\Braket{\cX}\subset M_n)$ for the tuples $\cX$ we work with. The crucial case is the base $n=3$ of the induction facilitated by \Cref{pr:ind.stp}, but the argument can be phrased uniformly for all $n\ge 3$. 

  In both cases the tuples $\cX$ we consider will consist of $n-1$ self-adjoint traceless diagonal matrices $\cD=\left(T_i\right)_{i=1}^{n-1}$ in $M_n$ (thus spanning the traceless part $\Braket{\cD}$ of the diagonal algebra $D_n$), supplemented by one other matrix $X$. We may as well assume the $T_i$ and $X$ mutually orthogonal with respect to the inner product defined by $\tau$.
   
  $X$ will be generic subject to that orthogonality requirement, so in particular with both $X$ and $X^*X=X^2$ having simple spectrum and being diagonal with respect to an orthonormal basis $(f_j)$ in general position with respect to the standard $(e_j)$ (shorthand phrasing such as \emph{$(f_j)$-diagonal} always has the obvious meaning). We will also have occasion to work with the corresponding matrix units $f_{ij}:=\Ket{f_i}\Bra{f_j}$.

  Note that the degree matrix attached to $X$ is $(f_j)$-diagonal. Since on the other hand the tuple $(T_i)$ contributes a scalar to the overall degree matrix $D$ of $\cX$, $D$ generates the same commutative $C^*$-algebra as $X$. For that reason, $D$ and $X$ are both $\G$-fixed. It follows that $\G$ leaves invariant the orthogonal complement $X^{\Braket{\cX}\perp}=\Braket{\cD}$ in $\Braket{\cX}$.

  All in all, $\G$ operates on $\Braket{\cD}$, $(f_j)$-diagonally on $M_n$ because it fixes $D$. Were that $(f_j)$-diagonal $\G$-action on $M_n$ non-trivial, it would effect a splitting
  \begin{equation*}
    M_n = M_{n,F}\oplus M_{n,F'}
    ,\quad
    M_{n,\bullet}:=\spn\left\{f_{ij}\ :\ (i,j)\in \bullet\right\}
    ,\quad
    F':=[n]^2\setminus F
  \end{equation*}
  with both $F$ and $F'$ non-empty. The generic position of $(e_j)$ and $(f_j)$, though, ensures that no such decomposition can be compatible with $\Braket{\cD}$:
  \begin{equation*}
    F,F'\ne \emptyset
    \xRightarrow{\quad}
    \Braket{\cD}
    \ne
    \left(\Braket{\cD}\cap M_{n,F}\right)
    \oplus
    \left(\Braket{\cD}\cap M_{n,F'}\right).
  \end{equation*}
  This contradicts $\Braket{\cD}$'s $\G$-invariance, finishing the proof. 
\end{proof}

\begin{lemma}\label{le:nd34}
  $\tensor*[_{3}]{\cat{qRig}}{_4}$ holds.
\end{lemma}
\begin{proof}
  Again set $\G:=\Qut(\Braket{\cX}\subset M_n)$. We borrow much from the classical version of the claim, proven as part of \cite[Proposition 3.21]{zbMATH07502493}. Consider tuples $\cX$ consisting of
  \begin{itemize}[wide]
  \item two traceless diagonal matrices $T_i$, $i=1,2$ (which thus span the traceless diagonal subspace of $M_3$);

  \item and two matrices of the form
    \begin{equation*}
      Y_i:=
      \begin{pmatrix}
        0&\Ket{\alpha_i}\\
        \Bra{\alpha_i}&0
      \end{pmatrix}
      ,\quad
      i=1,2,
    \end{equation*}
    with $\Ket{-}$ and $\Bra{-}$ indicating column and row vectors respectively. 
  \end{itemize}
  Hilbert-Schmidt orthonormality for the $Y_i$ means
  \begin{equation*}
    \frac 23
    \mathrm{Re}\Braket{\alpha_i\mid \alpha_j}
    =
    \delta_{ij}
    ,\quad
    1\le i,j\le 2.
  \end{equation*}
  We work with the (scaled) contribution 
  \begin{equation*}
    D
    :=
    \sum_i Y_i^* Y_i
    =
    \begin{pmatrix}
      \sum_i\Ket{\alpha_i}\Bra{\alpha_i} & 0\\
      0 & \sum_i\Braket{\alpha_i\mid\alpha_i}
    \end{pmatrix}
    =
    \begin{pmatrix}
      D'&0\\
      0&\lambda_0
    \end{pmatrix}
  \end{equation*}
  of $Y_i$ to the degree matrix $D_{\Braket{\cX}}$, for that of the $T_i$ is in any case a scalar. $D$'s eigenvalues/vectors are
  \begin{center}
    \begin{tabular}{|c|c|c|}
      $\lambda_0=3$&$\lambda_1=\frac 32(1-c)$&$\lambda_2=\frac 32(1+c)$\\
      \hline        
      $\begin{pmatrix}
        1\\
        0\\
        0
      \end{pmatrix}$
      &
        $\begin{pmatrix}
          \Ket{\alpha_1}+i\Ket{\alpha_2}\\       
          0
        \end{pmatrix}$
      &
        $\begin{pmatrix}
          \Ket{\alpha_1}-i\Ket{\alpha_2}\\       
          0
        \end{pmatrix}$\\
    \end{tabular}
  \end{center}
  for $\Braket{\alpha_1\mid \alpha_2}=\frac 32 ci\in \bR i$, so in particular the generic $D$ is invertible and the eigenvalues $\lambda_j\lambda_k^{-1}$ of $\Ad_D$ (conjugation by $D$ on $M_3$) are distinct. Because
  \begin{equation}\label{eq:yx.kers}
    \begin{aligned}
      \bC \left(E_+:=Y_1+i Y_2\right)
      &=
        \Braket{\cX}
        \cap
        \left(
        \sum_{\eta\in \left\{\lambda_1\lambda_0^{-1},\ \lambda_2^{-1}\lambda_0\right\}}
        \ker\left(\eta-\Ad_D\right)
        \right)\\
      \bC \left(E_-:=Y_1-i Y_2\right)
      &=
        \Braket{\cX}
        \cap
        \left(
        \sum_{\eta\in \left\{\lambda_2\lambda_0^{-1},\ \lambda_1^{-1}\lambda_0\right\}}
        \ker\left(\eta-\Ad_D\right)
        \right)\\
    \end{aligned}    
  \end{equation}
  the lines $\bC E_{\pm}$ are both $\G$-invariant; so too, then, is their span $\Braket{E_i}_i=\Braket{Y_i}_i$ and its orthogonal complement $\Braket{T_i}_i$ in $\Braket{\cX}$. 

  Generic choices of $\alpha_i$ will ensure that $D$ and the $T_i$ generate the block-triangularly-embedded $M_2\times \bC\le M_3$, which is thus also $\G$-invariant. The fixed subalgebra $M_2^{\G}$ of its upper left-hand $2\times 2$ block consists of at least the commutant of $D'$ (maximal abelian in $M_2$). The $\G$-invariance of the span
  \begin{equation*}
    \Braket{T'_i}\le M_2
    ,\quad
    T'_i:=\text{upper $2\times 2$ block of $T_i$}
  \end{equation*}
  implies, because the $T'_i$ are in general position with respect to $D'$, that in fact $M_2^{\G}=M_2$ and hence $\G$ fixes the $T_i$. The $\G$-action on $M_3$ thus amounts is a grading
  \begin{itemize}[wide]
  \item fixing the block-diagonal algebra $M_2\times \bC\le M_3$ pointwise;

  \item and operating on its orthogonal complement $\Braket{E_{\pm}}=\Braket{Y_i}$ by
    \begin{equation*}
      E_{\pm}
      \xmapsto{\quad \text{for }E_-=E_+^*\quad}
      E_{\pm}\otimes \delta^{\pm 1}
      \quad\text{for a group-like $\delta\in \cO(\G)$}.
    \end{equation*}     
  \end{itemize}
  The generically-non-zero $E_+^2$ belonging to the fixed-point algebra $M_3^{\G}$, we have $\delta^2=1$ and hence the $\G$-action factors through the classical $\bZ/2$-action given by $\Ad_{\mathrm{diag}(1,1,-1)}$.

  \Cref{pr:gen.z2} below now shows that the generic action of $\G$ on $M_n$ is that induced by a $\bZ/2$-grading on $\bC^3$, so in particular classical; the conclusion thus follows from \cite[Theorem 3.19]{zbMATH07502493}. 
\end{proof}

% % OLD: BEFORE OUTSOURCING TO \Cref{pr:gen.z2}
% % 
% % The endomorphism algebra $\End_{\G}\left(M_3\right)$ of the $\G$-representation $M$ is either $M_9$ (if the action is trivial) or $M_5\times M_4$. \Cref{th:gen2blocks} then implies that for generic $\cX$ we have
% % \begin{equation*}
% %   \End_{\G}\left(M_3\right)\cong \prod_i M_{n_i}
% %   ,\quad
% %   n_i\ge 4,
% % \end{equation*}
% % and hence
% % \begin{equation}\label{eq:m9m5m4}
% %   \End_{\G}\left(M_3\right)
% %   \cong
% %   M_9\text{ or }M_5\times M_4.
% % \end{equation}
% % Since we also know (\Cref{cor:gen.diag}) that generic actions are essentially diagonal and hence gradings, they must be $\bZ/2$-gradings because the algebras \Cref{eq:m9m5m4} consist of at most two blocks. This, however, reduces the problem to its classical version, and we are done by \cite[Theorem 3.19]{zbMATH07502493}. 
% % 

The following result is a variant of \cite[Theorem A]{2505.07485v1}, and is presumably of some independent interest. We denote by $\bG(d,V)$ the \emph{Grassmannian} \cite[\S 5.1]{ms_nonl} of dimension-$d$ subspaces of $V$. 

\begin{theorem}\label{th:gen2blocks}
  Let $\G\circlearrowright V$ be a finite-dimensional representation of a compact quantum group and denote by $\G_W\le \G$ the isotropy quantum group of a subspace $W\le V$.

  \begin{enumerate}[(1),wide]
  \item For $d,\ell\in \bZ_{>0}$ the subset
    \begin{equation}\label{eq:lg.ends}
      \left\{
        W\in \bG(d,V)
        \ :\ 
        \End_{\G_W}V\cong \prod_i M_{n_i}
        ,\ \forall i\left(n_i\ge \ell\right)
      \right\}
      \subseteq
      \bG(d,V)
    \end{equation}
    is Zariski-open, so in particular dense if non-empty.

  \item The same holds of
    \begin{equation}\label{eq:lg.1end}
      \left\{
        W\in \bG(d,V)
        \ :\ 
        \End_{\G_W}V\cong \prod_i M_{n_i}
        ,\ \exists i\left(n_i\ge \ell\right)
      \right\}
      \subseteq
      \bG(d,V)
    \end{equation}
  \end{enumerate}
\end{theorem}
\begin{proof}
  Density follows from non-emptiness from the \emph{irreducibility} \cite[Theorem 5.4]{ms_nonl} of the Grassmannian. We focus mainly on the first claim, indicating afterwards the small alterations needed to obtain the second.

  \begin{enumerate}[(1),wide]
  \item Recall \cite[Definition 2.2.1]{agpr_pi} that a \emph{polynomial identity (PI)} for a ring $R$ is a non-commutative polynomial returning 0 no matter how elements of $R$ are substituted for the variables. It follows from \cite[Theorem 10.3.2, Lemma 10.3.1 and Proposition 10.2.2]{agpr_pi} that the simple quotients of a product $A=\prod_i M_{n_i}$ are all of dimension $\ge \ell^2$ precisely when $A$ is generated, as an ideal, by the images of the PIs of $M_{\ell-1}$.
  
  In the spirit of Tannaka-Krein duality, the representation category of $\G_{W}$ is generated by the orthogonal projection $P_{W}: V \mapsto W$ and the representation category of $\G$, closed under duality, tensor products etc.

    The condition that $W\in \bG(d,V)$ belong to \Cref{eq:lg.ends} is now expressible as follows:
    \begin{itemize}[wide]
    \item for some PIs $f_i$ of $M_{\ell-1}$;

    \item and elements $a_i$, $b_i$ and $x_{ij}$ in the algebra generated by the orthogonal projection $V\xrightarrowdbl{P_W} W$ and
      \begin{equation*}
        \bigcup_{N\in \bZ_{\ge 0}}
        \End_{\G}\left(\bC\oplus (V\oplus V^*)\oplus\cdots\oplus (V\oplus V^*)^{\otimes N}\right)
      \end{equation*}
      (generation including algebraic operations along with duality and tensoring);

    \item the element

      % % \begin{equation*}
      % %   \sum_i
      % %   P_W a_i P_W
      % %   \cdot
      % %   f_i\left(P_W x_{ij}P_W\right)_j
      % %   \cdot
      % %   P_W b_i P_W
      % %   \in
      % %   \End(W)
      % % \end{equation*}
      % % 

      \begin{equation*}
        P_V\left(
          \sum_i
          a_i 
          \cdot
          f_i\left( x_{ij}\right)_j
          \cdot
          b_i
        \right)
        P_V
        \in
        \End(V)
      \end{equation*}
      is invertible. 
    \end{itemize}
    Indeed, the last condition says that the ideal generated by the images of the PIs of $M_{l-1}$ contains an invertible element, which is equivalent to saying that the ideal is equal to the whole algebra. As invertibility is an open condition, the conclusion follows.
    
  \item $\End_{\G_W}V$ is in any case a finite product of matrix algebras, and has at least one simple factor of dimension $\ge \ell^2$ precisely when it does \emph{not} satisfy all PIs of $M_{\ell-1}$. This is again an open condition, so the preceding argument replicates. 
  \end{enumerate}
\end{proof}

Even though superseded by generic rigidity, the following principle (employed in the proof of \Cref{le:nd34}) can presumably be of some use in similar circumstances. 

\begin{proposition}\label{pr:gen.z2}
  Let $3\le n$ and $1\le d\le n^2-2$. If the action of $\Qut(\cS\le M_n)$ factors through a $\bZ/2$-grading on $\bC^n$ for at least one operator $d$-system $\cS\le M_n$, then it does so generically.
\end{proposition}
\begin{proof}
  We already know (\Cref{cor:gen.diag}) that the action is generically essentially diagonal and hence (induced by) a $\Gamma$-grading on $\bC^n$ for some group $\Gamma$. If for some $\cS_0$ that group is contained in $\bZ_2$, we have
  \begin{equation*}
    \begin{aligned}
      \End_{\Qut(\cS_0\le M_n)}M_n
      &\cong
        M_{n_+}\oplus M_{n_-}\\
      n_{\pm}
      &:=\dim\left(\text{degree-$\pm 1$ component of the grading on $M_n$}\right).
    \end{aligned}
  \end{equation*}
  \Cref{th:gen2blocks} ensures that for generic $\cS\le M_n$ the analogous endomorphism algebra $\End_{\Qut(\cS\le M_n)}M_n$
  \begin{itemize}[wide]
  \item is a product of matrix algebras of dimension $\ge \min\left(n_+^2,\ n_-^2\right)$;

  \item at least one of which is of dimension $\ge \max\left(n_+^2,\ n_-^2\right)$
  \end{itemize}
  Because $n_++n_-=n^2$, the only possibilities are
  \begin{equation*}
    \End_{\Qut(\cS_0\le M_n)}M_n
    \cong
    M_{n_+}\oplus M_{n_-}
    \quad\text{or}\quad
    M_{n^2}.
  \end{equation*}
  There being at most two simple factors, this is a $\bZ/2$-grading. 
\end{proof}

\begin{proposition}\label{pr:n122}
  $\tensor*[_{4\le n}]{\cat{qRig}}{_{2\le d\le (n-1)^2-2}}$ holds.
\end{proposition}
\begin{proof}
  This is a simple application of \Cref{pr:ind.stp}: we know from \Cref{th:dn1n2} that $\tensor*[_{3\le m}]{\cat{qRig}}{_{2}}$ is valid, so by symmetry so is $\tensor*[_{3\le m}]{\cat{qRig}}{_{m^2-3}}$. $\tensor*[_{4\le m+1}]{\cat{qRig}}{_{m^2-2}}$ follows from \Cref{pr:ind.stp}, i.e. the sought-after conclusion with $m=n-1$.
\end{proof}

%%%%%%%%%%%%%%%%%%%%%%%%%%%%%%%%
%%%%%%%%%%%%%%%%%%%%%%%%%%%%%%%%
\section{Computer verification for small cases}\label{sec:comp.alg}
Here we will use and extend the fact that the degree matrix is fixed (see \Cref{prop:degreefixed}). In the same way one shows that $D_{2} := A^2\mathds{1}$ is fixed under the action of the quantum automorphism group. 
\begin{lemma}
Let $\mathcal{G}$ be a quantum graph. If the algebra generated by $D$ and $D_2$ is equal to $M_n$ then the quantum automorphism group of $\mathcal{G}$ is trivial.
\end{lemma}
\begin{proof}
The whole algebra generated by $D$ and $D_2$ is fixed by the action of the quantum automorphism group, so if this algebra is equal to $M_n$ then the action has to be trivial.
\end{proof}
\begin{proposition}
Let $n\geqslant \mathbb{N}$ and $d \in \{2,\dots, n^2-3\}$. We think of  a $d$-tuple $\cX$ of traceless Hermitian matrices as an operator $d$-system by adding the unit and taking the linear span, hence a quantum graph. For each $n\leqslant 8$ and $d \in \{2,\dots, n^2-3\}$ we can find a $d$-tuple such that the algebra generated by $D$ and $D_2$ is equal to $M_n$. More precisely, the matrices $(D^{i} D_{2}^{j})_{i,j \in \{0,\dots, n-1\}}$ form a basis of $M_n$.
\end{proposition}
\begin{proof}
Let $\cX = \{X_1,\dots, X_d\}$. Then we can write the formula for $D = \sum_{i,j=1}^{d} (\Gamma^{-1})_{ij} X_{i} X_{j}$, where $\Gamma_{ij}:= \tau(X_{i} X_{j})$ is the Gram matrix of the tuple $\cX$. Indeed, one can easily verify that the matrices $\widetilde{X_{j}}:= \sum_{i} (\Gamma^{-\frac{1}{2}})_{ij} X_{i}$ form an orthonormal basis of $\Braket{\cX}$, so the formula for $D$ follows from the standard formula $D = \sum_{j} \widetilde{X}_{j}^2$. We similarly obtain the formula for $D_2 = \sum_{i,j,k,l} (\Gamma^{-1})_{ij} (\Gamma^{-1})_{kl} X_{k} X_{i} X_{j} X_{l}$. Verification that the matrices $(D^{i} D_{2}^{j})_{i,j \in \{0,\dots, n-1\}}$ form a basis of $M_n$ can be done on a computer and this is exactly what we did. For instance, for $n=7$ and $d=4$ we have the following example:
\[
\resizebox{\textwidth}{!}{\ensuremath{
\begin{aligned}
X_1&=\begin{bmatrix}
  -0.142491 &  0.0251255 &    0.16238 &  0.0555206 &    0.175618 & -0.0938962 &  -0.0929846 \\
  0.0251255 &   0.167523 & -0.0250914 &  0.0559988 &  -0.0593764 &  0.0799041 &   -0.165092 \\
    0.16238 & -0.0250914 &  0.0679049 &   0.189596 &   0.0681018 &  0.0698056 &     0.28165 \\
  0.0555206 &  0.0559988 &   0.189596 &  -0.166191 &   -0.191064 &  0.0831966 &   0.0652995 \\
   0.175618 & -0.0593764 &  0.0681018 &  -0.191064 &    0.154533 &  0.0791232 &    0.176272 \\
 -0.0938962 &  0.0799041 &  0.0698056 &  0.0831966 &   0.0791232 &   0.186155 &    0.256028 \\
 -0.0929846 &  -0.165092 &    0.28165 &  0.0652995 &    0.176272 &   0.256028 &   -0.267433
\end{bmatrix},
\\
X_2&=\begin{bmatrix}
  0.0710277 &  -0.149662 &   0.200781 &  -0.167936 &   0.0346357 &  -0.290504 &   -0.109015 \\
  -0.149662 &  -0.120797 &  -0.206312 &  0.0435404 &     0.19233 &  0.0415093 &    0.026054 \\
   0.200781 &  -0.206312 & -0.0635957 & -0.0823084 &  -0.0903822 &   0.151709 &   -0.188229 \\
  -0.167936 &  0.0435404 & -0.0823084 &  -0.117987 &  -0.0552701 &    0.22579 &   -0.070853 \\
  0.0346357 &    0.19233 & -0.0903822 & -0.0552701 &     0.31879 & -0.0875183 &   -0.107513 \\
  -0.290504 &  0.0415093 &   0.151709 &    0.22579 &  -0.0875183 & -0.0229677 &  -0.0643047 \\
  -0.109015 &   0.026054 &  -0.188229 &  -0.070853 &   -0.107513 & -0.0643047 &  -0.0644701
\end{bmatrix},
\\
X_3&=\begin{bmatrix}
   0.150074 &   0.101414 & -0.0925115 &    0.23272 &    0.271942 & 0.00275495 &  -0.235919 \\
   0.101414 &   0.205297 &  0.0416222 &  -0.049573 &   0.0364604 &   0.172456 &   0.105198 \\
 -0.0925115 &  0.0416222 &   0.230496 & -0.0894059 &    0.126021 &  0.0379802 &  -0.103049 \\
    0.23272 &  -0.049573 & -0.0894059 &  0.0720574 &    0.100826 &  0.0606621 &  -0.158897 \\
   0.271942 &  0.0364604 &   0.126021 &   0.100826 &  -0.0963457 &  0.0647688 &  -0.125808 \\
 0.00275495 &   0.172456 &  0.0379802 &  0.0606621 &   0.0647688 &  -0.195335 &  -0.060983 \\
  -0.235919 &   0.105198 &  -0.103049 &  -0.158897 &   -0.125808 &  -0.060983 &  -0.366243
\end{bmatrix},
\\
X_4&=\begin{bmatrix}
  0.0152974 &  0.0445368 &   0.146435 &  0.0931627 &    0.115207 &  0.0059397 &   -0.114823 \\
  0.0445368 &  0.0559448 & -0.0651526 &   0.254707 &   0.0377037 &  0.0989529 &   -0.148023 \\
   0.146435 & -0.0651526 &   -0.14917 &   0.176034 &    0.122183 & 0.00138746 &    0.072744 \\
  0.0931627 &   0.254707 &   0.176034 &  0.0718361 &    0.116681 & -0.0276206 &  -0.0653339 \\
   0.115207 &  0.0377037 &   0.122183 &   0.116681 &   -0.144456 &  -0.308034 & -0.00932981 \\
  0.0059397 &  0.0989529 & 0.00138746 & -0.0276206 &   -0.308034 &  -0.301774 &  -0.0260149 \\
  -0.114823 &  -0.148023 &   0.072744 & -0.0653339 & -0.00932981 & -0.0260149 &    0.452321
\end{bmatrix}.
\end{aligned}
}}
\]
\end{proof}

This provides an alternative to the approach presented in \Cref{subsec:small.case}.

\begin{proposition}
  Let $n\geqslant \mathbb{N}$ and $d \in \{2,\dots, n^2-3\}$. If there exists a $d$-tuple $\cX$ of traceless Hermitian matrices such that the matrices $(D^{i} D_{2}^{j})_{i,j \in \{0,\dots, n-1\}}$ form a basis of $M_n$ then the same happens for almost all choices of such $d$-tuples.
\end{proposition} 
\begin{proof}
  Because of the formula $D = \sum_{i,j=1}^{d} (\Gamma^{-1})_{ij} X_{i} X_{j}$ and an analogous one for $D_2$, we see they are built from $\cX$ in an algebraic fashion. To be precise they are rational function, because $\Gamma^{-1}$ contains an inverse of the determinant, but we can replace $D$ by $\det(\Gamma) D$ to obtain a truly algebraic expression and it does not affect the condition we want $D$ and $D_2$ to satisfy. The matrices $(D^{i} D_{2}^{j})_{i,j \in \{0,\dots, n-1\}}$ form a basis of $M_n$ if and only if they are linearly independent, which can be stated as non-vanishing of a certain determinant. It means the set of $d$-tuples $\cX$ for which it holds is Zariski open, so it is dense as soon as it is non-empty, and Zariski open and dense subsets are of full measure.
\end{proof}
This simple proposition shows a possible strategy for proving triviality of the quantum automorphism group of a random quantum graph in a different manner.

%%%%%%%%%%%%%%%%%%%%%%%%%%%%%%%%
%%%%%%%%%%%%%%%%%%%%%%%%%%%%%%%%

\addcontentsline{toc}{section}{References}
%\bibliography{bib}{}

\begin{thebibliography}{10}

\bibitem{agpr_pi}
Eli Aljadeff, Antonio Giambruno, Claudio Procesi, and Amitai Regev.
\newblock {\em Rings with polynomial identities and finite dimensional
  representations of algebras}, volume~66 of {\em Colloq. Publ., Am. Math.
  Soc.}
\newblock Providence, RI: American Mathematical Society (AMS), 2020.

\bibitem{boll_rnd-grph_2001}
B\'ela Bollob\'as.
\newblock {\em Random graphs}, volume~73 of {\em Cambridge Studies in Advanced
  Mathematics}.
\newblock Cambridge University Press, Cambridge, second edition, 2001.

\bibitem{brcx_hndbk-1}
Francis Borceux.
\newblock {\em Handbook of categorical algebra. {Volume} 1: {Basic} category
  theory}, volume~50 of {\em Encycl. Math. Appl.}
\newblock Cambridge: Cambridge Univ. Press, 1994.

\bibitem{MR4139106}
G.~Boreland, I.~G. Todorov, and A.~Winter.
\newblock Sandwich theorems and capacity bounds for non-commutative graphs.
\newblock {\em J. Combin. Theory Ser. A}, 177:Paper No. 105302, 39, 2021.

\bibitem{blr_neron}
Siegfried Bosch, Werner L\"utkebohmert, and Michel Raynaud.
\newblock {\em N\'eron models}, volume~21 of {\em Ergebnisse der Mathematik und
  ihrer Grenzgebiete (3) [Results in Mathematics and Related Areas (3)]}.
\newblock Springer-Verlag, Berlin, 1990.

\bibitem{zbMATH07202497}
Michael Brannan, Alexandru Chirvasitu, Kari Eifler, Samuel Harris, Vern
  Paulsen, Xiaoyu Su, and Mateusz Wasilewski.
\newblock Bigalois extensions and the graph isomorphism game.
\newblock {\em Commun. Math. Phys.}, 375(3):1777--1809, 2020.

\bibitem{bcehpsw_bigal}
Michael Brannan, Alexandru Chirvasitu, Kari Eifler, Samuel Harris, Vern
  Paulsen, Xiaoyu Su, and Mateusz Wasilewski.
\newblock Bigalois extensions and the graph isomorphism game.
\newblock {\em Commun. Math. Phys.}, 375(3):1777--1809, 2020.

\bibitem{MR4507619}
Michael Brannan, Samuel~J. Harris, Ivan~G. Todorov, and Lyudmila Turowska.
\newblock Synchronicity for quantum non-local games.
\newblock {\em J. Funct. Anal.}, 284(2):Paper No. 109738, 54, 2023.

\bibitem{MR4140642}
Javier~Alejandro Ch\'avez-Dom\'inguez and Andrew~T. Swift.
\newblock Connectivity for quantum graphs.
\newblock {\em Linear Algebra Appl.}, 608:37--53, 2021.

\bibitem{2505.07485v1}
Alexandru Chirvasitu.
\newblock Generically-constrained quantum isotropy, 2025.
\newblock \url{http://arxiv.org/abs/2505.07485v1}.

\bibitem{zbMATH07502493}
Alexandru Chirvasitu and Mateusz Wasilewski.
\newblock Random quantum graphs.
\newblock {\em Trans. Am. Math. Soc.}, 375(5):3061--3087, 2022.

\bibitem{MR376726}
Man~Duen Choi.
\newblock Completely positive linear maps on complex matrices.
\newblock {\em Linear Algebra Appl.}, 10:285--290, 1975.

\bibitem{zbMATH07856619}
Matthew Daws.
\newblock Quantum graphs: different perspectives, homomorphisms and quantum
  automorphisms.
\newblock {\em Commun. Am. Math. Soc.}, 4:117--181, 2024.

\bibitem{MR4706978}
Matthew Daws.
\newblock Quantum graphs: different perspectives, homomorphisms and quantum
  automorphisms.
\newblock {\em Commun. Am. Math. Soc.}, 4:117--181, 2024.

\bibitem{dk_cqg}
Mathijs~S. Dijkhuizen and Tom~H. Koornwinder.
\newblock {CQG} algebras: {A} direct algebraic approach to compact quantum
  groups.
\newblock {\em Lett. Math. Phys.}, 32(4):315--330, 1994.

\bibitem{MR3015725}
Runyao Duan, Simone Severini, and Andreas Winter.
\newblock Zero-error communication via quantum channels, noncommutative graphs,
  and a quantum {L}ov\'asz number.
\newblock {\em IEEE Trans. Inform. Theory}, 59(2):1164--1174, 2013.

\bibitem{MR156334}
P.~Erd\H~os and A.~R\'enyi.
\newblock Asymmetric graphs.
\newblock {\em Acta Math. Acad. Sci. Hungar.}, 14:295--315, 1963.

\bibitem{egno}
P.~Etingof, S.~Gelaki, D.~Nikshych, and V.~Ostrik.
\newblock {\em Tensor categories}, volume 205 of {\em Mathematical Surveys and
  Monographs}.
\newblock American Mathematical Society, Providence, RI, 2015.

\bibitem{zbMATH06617237}
Pierre Fima and Lorenzo Pittau.
\newblock The free wreath product of a compact quantum group by a quantum
  automorphism group.
\newblock {\em J. Funct. Anal.}, 271(7):1996--2043, 2016.

\bibitem{zbMATH07632578}
Daniel Gromada.
\newblock Some examples of quantum graphs.
\newblock {\em Lett. Math. Phys.}, 112(6):49, 2022.
\newblock Id/No 122.

\bibitem{hck_tens_2e_2019}
Wolfgang Hackbusch.
\newblock {\em Tensor spaces and numerical tensor calculus}, volume~56 of {\em
  Springer Ser. Comput. Math.}
\newblock Cham: Springer, 2nd revised edition edition, 2019.

\bibitem{hj_mtrx}
Roger~A. Horn and Charles~R. Johnson.
\newblock {\em Matrix analysis.}
\newblock Cambridge: Cambridge University Press, 2nd ed. edition, 2013.

\bibitem{ksv_randomgr}
Jeong~Han Kim, Benny Sudakov, and Van~H. Vu.
\newblock On the asymmetry of random regular graphs and random graphs.
\newblock {\em Random Struct. Algorithms}, 21(3-4):216--224, 2002.

\bibitem{MR4481115}
Junichiro Matsuda.
\newblock Classification of quantum graphs on {$M_2$} and their quantum
  automorphism groups.
\newblock {\em J. Math. Phys.}, 63(9):Paper No. 092201, 34, 2022.

\bibitem{ms_nonl}
Mateusz Micha{\l}ek and Bernd Sturmfels.
\newblock {\em Invitation to nonlinear algebra}, volume 211 of {\em Grad. Stud.
  Math.}
\newblock Providence, RI: American Mathematical Society (AMS), 2021.

\bibitem{mrv_qtm-compos}
Benjamin Musto, David Reutter, and Dominic Verdon.
\newblock A compositional approach to quantum functions.
\newblock {\em J. Math. Phys.}, 59(8):081706, 42, 2018.

\bibitem{MR3849575}
Benjamin Musto, David Reutter, and Dominic Verdon.
\newblock A compositional approach to quantum functions.
\newblock {\em J. Math. Phys.}, 59(8):081706, 42, 2018.

\bibitem{mrv_qtm-qgriso}
Benjamin Musto, David Reutter, and Dominic Verdon.
\newblock The {Morita} theory of quantum graph isomorphisms.
\newblock {\em Commun. Math. Phys.}, 365(2):797--845, 2019.

\bibitem{NeTu13}
Sergey Neshveyev and Lars Tuset.
\newblock {\em Compact quantum groups and their representation categories},
  volume~20 of {\em Cours Sp\'ecialis\'es [Specialized Courses]}.
\newblock Soci\'et\'e Math\'ematique de France, Paris, 2013.

\bibitem{pls-bk}
Vern Paulsen.
\newblock {\em Completely bounded maps and operator algebras}, volume~78 of
  {\em Cambridge Studies in Advanced Mathematics}.
\newblock Cambridge University Press, Cambridge, 2002.

\bibitem{MR3114209}
Vern~I. Paulsen and Fred Shultz.
\newblock Complete positivity of the map from a basis to its dual basis.
\newblock {\em J. Math. Phys.}, 54(7):072201, 12, 2013.

\bibitem{rad}
David~E. Radford.
\newblock {\em Hopf algebras}, volume~49 of {\em Series on Knots and
  Everything}.
\newblock World Scientific Publishing Co. Pte. Ltd., Hackensack, NJ, 2012.

\bibitem{MR3447998}
Dan Stahlke.
\newblock Quantum zero-error source-channel coding and non-commutative graph
  theory.
\newblock {\em IEEE Trans. Inform. Theory}, 62(1):554--577, 2016.

\bibitem{mw_einstein-summation}
Christopher Stover and Eric~W Weisstein.
\newblock Einstein summation, 2025.
\newblock available at
  \url{https://mathworld.wolfram.com/EinsteinSummation.html} (accessed
  2025-05-19).

\bibitem{tim}
Thomas Timmermann.
\newblock {\em An invitation to quantum groups and duality}.
\newblock EMS Textbooks in Mathematics. European Mathematical Society (EMS),
  Z\"{u}rich, 2008.
\newblock From Hopf algebras to multiplicative unitaries and beyond.

\bibitem{zbMATH00908215}
Alfons Van~Daele and Shuzhou Wang.
\newblock Universal quantum groups.
\newblock {\em Int. J. Math.}, 7(2):255--263, 1996.

\bibitem{MR4482713}
Dominic Verdon.
\newblock A covariant {S}tinespring theorem.
\newblock {\em J. Math. Phys.}, 63(9):Paper No. 091705, 49, 2022.

\bibitem{MR4707042}
Dominic Verdon.
\newblock Covariant quantum combinatorics with applications to zero-error
  communication.
\newblock {\em Comm. Math. Phys.}, 405(2):Paper No. 51, 57, 2024.

\bibitem{zbMATH00763766}
Shuzhou Wang.
\newblock Free products of compact quantum groups.
\newblock {\em Commun. Math. Phys.}, 167(3):671--692, 1995.

\bibitem{zbMATH01316946}
Shuzhou Wang.
\newblock Quantum symmetry groups of finite spaces.
\newblock {\em Commun. Math. Phys.}, 195(1):195--211, 1998.

\bibitem{zbMATH06008057}
Nik Weaver.
\newblock Quantum relations.
\newblock In {\em A von Neumann algebra approach to quantum metrics / Quantum
  relations}, pages 81--140. Providence, RI: American Mathematical Society
  (AMS), 2012.

\end{thebibliography}
%\bibliographystyle{plain}

\def\polhk#1{\setbox0=\hbox{#1}{\ooalign{\hidewidth
  \lower1.5ex\hbox{`}\hidewidth\crcr\unhbox0}}}

\Addresses

\end{document}